\documentclass[11pt,a4paper]{amsart}

\title{Automorphism groups of Matsuo algebras}
\author{Jari Desmet}
\address{\parbox{\linewidth}{Ghent University \\ Department of Mathematics, Computer Science and Statistics \\ Krijgslaan 299 -- S9 \\ 9000 Gent \\ Belgium}}
\email{\href{mailto:jari.desmet@ugent.be}{jari.desmet@ugent.be}}
\date{\today}
\makeatletter
\@namedef{subjclassname@2020}{%
  \textup{2020} Mathematics Subject Classification}
\makeatother

\usepackage[english]{babel}
\usepackage{amsmath,amssymb,amsthm,mathtools}
\usepackage[utf8]{inputenc}
\usepackage[margin=3cm]{geometry}
\usepackage{hyperref}
\usepackage[capitalize,noabbrev]{cleveref}
\usepackage{todonotes}
\usepackage{enumitem}
\usepackage{verbatim}

\DeclareMathOperator{\kar}{char}
\DeclareMathOperator{\ad}{ad}
\DeclareMathOperator{\Spec}{Spec}

\DeclareMathOperator{\Wr}{Wr}
\DeclareMathOperator{\im}{im}

\DeclareMathOperator{\ZS}{ZS}
\newcommand{\F}{k}
\newcommand{\Z}{\mathbb{Z}}

\DeclareMathOperator{\Der}{Der}
\newcommand{\SU}{\mathrm{SU}}

\newcommand{\Aut}{\mathrm{Aut}}

\setlist[enumerate]{label = {\rm (\roman*)}}

\newtheorem{theorem}{Theorem}[section]
\newtheorem{corollary}[theorem]{Corollary}
\newtheorem{lemma}[theorem]{Lemma}
\newtheorem{proposition}[theorem]{Proposition}

\theoremstyle{definition}
\newtheorem{definition}[theorem]{Definition}

\theoremstyle{remark}

\newtheorem*{remark}{Remark}
\newtheorem{example}[theorem]{Example}

\newcommand{\DATwoDrawing}[6]{
		\draw (0,0)
		-- (0.5,1) node[circle,fill=black,label=-3:{#3}]  {}
		-- ++(1,2) node[circle,fill=black,label=3:{#2}]  {}
		-- ++(1,2) node[circle,fill=black,label=right:{#1}]  {}
		-- ++(1,-2) node[circle,fill=black,label=3:{#6}]  {}
		-- ++(0.5,-1) node[circle,fill=black,label=-3:{#5}]  {}
		-- ++(1,-2);
		\draw (0,3)
		-- ++(7.5,0) node[circle,fill=black,label=above:{#4}]  {}
		-- ++(-7.5,-7.5/3.5);
}

\newcommand{\AThreeDrawing}[9]{
		\useasboundingbox	(-1.0cm,1.5cm) rectangle (5.0cm,-5.3cm);
		\draw (0,0) node (1) [label=left:{#1}] {}
			++(0:2cm) node (2) [label=above:{#2}] {}
			++(0:2cm) node (3) [label=right:{#3}] {}
			++(270:2cm) node (6) [label=right:{#6}] {}
			++(180:2cm) node (5) [label=above right:{#5}] {}
			++(180:2cm) node (4) [label=left:{#4}] {}
			++(270:2cm) node (7) [label=left:{#7}] {}
			++(0:2cm) node (8) [label=below:{#8}] {}
			++(0:2cm) node (9) [label=right:{#9}] {}
		;

		\draw (1) -- (2) -- (3);
		\draw (4) -- (5) -- (6);
		\draw (7) -- (8) -- (9);
		\draw (1) -- (4) -- (7);
		\draw (2) -- (5) -- (8);
		\draw (3) -- (6) -- (9);
		\draw (1) -- (5) -- (9);
		\draw (3) -- (5) -- (7);

		\draw (1) to (6);
		\draw[shift=(1)] (6) .. controls (332:7.5cm) and (298:7.5cm)	.. (8);
		\draw (8) to (1);

		\draw (3) to (4);
		\draw[shift=(3)] (4) .. controls (208:7.5cm) and (242:7.5cm)	.. (8);
		\draw (8) to (3);

		\draw (7) to (2);
		\draw[shift=(7)] (2) .. controls (62:7.5cm) and (28:7.5cm)	.. (6);
		\draw (6) to (7);

		\draw (9) to (2);
		\draw[shift=(9)] (2) .. controls (118:7.5cm) and (152:7.5cm)	.. (4);
		\draw (4) to (9);
}

\begin{document}
\maketitle	
\begin{abstract}
	In this paper, we determine the connected component of the automorphism group scheme of Matsuo algebras over fields of characteristic not $3$.
\end{abstract}
\section{Introduction}
Matsuo algebras are a class of non-associative algebras related to $3$-transposition groups that first arose in the context of \emph{vertex operator algebras} in \cite{matsuo}. This class of algebras inspired the notion of \emph{axial algebras}. Motivated by the classification of $2$-generated Norton-Sakuma algebras appearing in the Griess algebras of OZ vertex operator algebras (\cite{sakuma}), Jonathan Hall, Felix Rehren and Sergey Shpectorov introduced this framework in \cite{universalaxial} as a new way of studying phenomena related to both Griess algebras and $3$-transposition groups. The aforementioned authors showed in \cite{axialjordan} that the axial algebras of Jordan type $\eta \neq \tfrac{1}{2}$ are quotients of Matsuo algebras, but for $\eta=\tfrac{1}{2}$, these algebras have not been classified yet. 
	
	In an effort to tackle this problem, Sergey Shpectorov introduced the concept of \emph{solid subalgebras}, $2$-generated subalgebras of an axial algebra $A$ for which all primitive idempotents of this subalgebra are again axes for the whole algebra $A$. They first appeared in the work \cite{JustinSergey} by Sergey Shpectorov and Justin McInroy, and were used by Ilya Gorshkov, Alexei Staroletov, Sergey Shpectorov and the author in \cite{gorshkov2024solid,solidisjordan} to show that (quotients of) Matsuo algebras are the only primitive axial algebras of Jordan type with finite automorphism groups. 
	
	In Jordan algebras, any $2$-generated subalgebra is solid (\cite[Chapter III, Section 1, Lemma 1]{jacobson}), and conversely, any axial algebra in which all $2$-generated subalgebras are solid is automatically Jordan (\cite[Theorem~1.3]{solidisjordan}). Those Matsuo algebras which are also Jordan algebras where classified by Tom De Medts and Felix Rehren in \cite{felixtom} (with a correction in characteristic $3$ by Takahiro Yabe in \cite{char3}), showing that solid subalgebras are a rather rare occurence in Matsuo algebras. However, an example of a non-Jordan Matsuo algebra containing solid lines was found by Gorshkov, Shpectorov and Staroletov in \cite{gorshkov2024solid}, namely $M_{\frac{1}{2}}(3^3\colon S_4)$. To further explore the dichotomy between Matsuo algebras and Jordan algebras, we can classify solid subalgebras in Matsuo algebras by determining their automorphism groups, using \cite[Theorem~1.2]{solidisjordan}. In this paper, we show that, generally, Matsuo algebras tend to have a finite automorphism group (always when $\eta \neq \tfrac{1}{2}$), while this is obviously not true for Jordan algebras.
	
	The results in this paper resolve the questions posed in \cite[Questions 7.3, 7.4 and 7.5]{gorshkov2024solid} at least partially. In \cite{gorshkov2024solid}, the authors ask whether it is possible to classify the (quotients of) Matsuo algebras that contain non-trivial solid subalgebras, and whether the example $M_{\frac{1}{2}}(3^3\colon S_4)$ is part of an infinite family. Indeed, the answer to these questions is a resounding yes when $\kar k \neq 3$, by \cref{thm:autgroup3nW}. The automorphism group of $M_{\frac{1}{2}}(3^n\colon W)$ acts transitively on all primitive idempotents contained in the so-called vertical lines (see \cref{lem:lineorbits} for a definition), while the automorphism group can only permute the vertical lines discretely, hence the horizontal lines are not solid. Our classification also shows that the only Matsuo algebras containing solid subalgebras are of type $S_n$ and $3^n\colon W$ when $k\neq 3$. When on the other hand the characteristic of the field is equal to $3$, the solutions to these questions are much less clear, and novel techniques would be required to handle those cases, but we believe a similar approach does yield interesting results (for example, \cref{lem:planarD4} does not require $\kar k \neq 3$). To handle quotients of Matsuo algebras\index{Matsuo algebra}, we are less optimistic a similar approach would work. The set $D$ is then no longer a basis of the algebra, and might behave in a less rigid way. 
		
		\subsection{Outline of the paper}
		To determine the connected component of the automorphism group of an algebra $A$ we consider the Lie algebra of derivations $\Der(A)$. Indeed, we can view the automorphism group of a finite-dimensional algebra as a group scheme $\mathbf{Aut}(A)$, and then the Lie algebra of derivations gives a model for the Lie algebra of the group scheme \cite[Corollary 10.7]{humphreyslag}. Hence, whenever $\Der(A)$ is trivial, the automorphism group $\mathbf{Aut}(A)$ is necessarily finite.
		
		Checking whether a linear map is a derivation of an algebra $A$ is a linear problem involving the structure constants of $A$. However, with respect to the canonical basis, the table of structure constants of Matsuo algebras is very sparse with respect to the canonical basis, which ensures that the conditions for a linear map to be a derivation of the Matsuo algebra are also sparse. This will allow us to solve the linear system of equations by hand, hence giving upper bounds on the dimension of $\Der(A)$.
		
		In \cref{sec:planarrel}, we show that most (but not all!) conditions for a linear map to be a derivation have support in just a $3$-generated subalgebra of the algebra $A$ (we explain what we mean with this below). This will allow to check (non-)existence of derivations locally.
		
		The ``planar'' relations found in \cref{sec:planarrel} will be exploited in \cref{sec:4gen} to see whether there can exist derivations in each of the $4$-generated Matsuo algebras, i.e.\@ the Matsuo algebras of central type $A_5$, $3^3 \colon S_4 $, $D_4$, $3^3\colon 2$, $[3^{10}]\colon 2$ and $2^6\colon \SU_3(2)'$. It turns out that the condition $\kar \F\neq 3$ to conclude something meaningful is essential in this section, since $4$-generated Matsuo algebras tend to have much larger derivation algebras when $\kar \F = 3$. 
		
		We know by \cref{sec:planarrel} that ``planar'' relations have support inside $4$-generated Matsuo algebras, and by \cref{sec:4gen} that they (locally) do not allow for non-zero derivations in most $4$-generated Matsuo algebras. This means we can use the combinatorics of $3$-transposition groups and Fischer spaces to show that most Matsuo algebras have trivial derivation algebra. Concretely, in \cref{sec:classFischerspaces} we show that there exist only two infinite families which might have non-trivial derivations. The first of the two is related to Jordan algebras, while the second of the two is related to the example found by Gorshkov, Shpectorov and Staroletov in \cite{gorshkov2024solid}. 
		
		Lastly, in \cref{sec:aut3nW}, we explicitly determine the full identity component of the automorphism groups of the two infinite families which have non-trivial derivations. We conclude that the dimensions of $\Der(A)$ and $\mathbf{Aut}(A)$ coincide, showing that when $\kar \F\neq 3$, the automorphism group of a Matsuo algebra is always smooth, and often trivial.
		
%
\subsection{Notation} Throughout, $k$ will denote a field of $\kar k \neq 2$. From \cref{sec:4gen} onwards, we will assume $\kar k \neq 3$. To differentiate between group schemes and abstract groups, we will write group schemes in boldface.

\subsection{Acknowledgements} This research was done as part of the author's PhD, with financial support from the FWO PhD mandate 1172422N. We are grateful to Tom De Medts for his guidance during this research.
\section{Preliminaries}
\subsection{Axial algebras}To define axial algebras, we need to first introduce the notion of a \emph{fusion law}.
\begin{definition}
	A fusion law $(\mathcal{F},\star)$ is a set with a map $\star \colon \mathcal{F}\times\mathcal{F} \to 
2^{\mathcal{F}}$. Here $2^\mathcal{F}$ denotes the set of all subsets of $\mathcal{F}$.
	\end{definition}

%
	
	A fusion law can be denoted using a multiplication table, in the same way any binary operation can be. As a convention, if the product of $a,b\in \mathcal{F}$ is the empty set, we leave the entry empty. We also omit the set brackets $\{\}$ in the entries.
	
	Probably the two most important examples of fusion laws are the  Jordan fusion law $\mathcal{J}(\eta)$ and the Monster fusion law $\mathcal{M}(\alpha,\beta)$ (see \cref{table:jordanmonster}). The Jordan type fusion law with $\eta=\tfrac{1}{2}$ first appeared in the structure theory of Jordan algebras (see e.g.\@ \cite[Section 9]{jacobson}), while the Monster fusion law has connections to the Griess algebra, the Monster group and vertex operator algebras (see e.g.\@ \cite{ivanov2009monster}). For a recent survey on these two fusion laws, see \cite{JustinSergey}.
	\begin{table}[h]
		\[
		\begin{array}{c || c | c | c}
			\star & 1&0 &\eta \\
		\hline
		\hline
		1& 1 &  & \eta\\
		\hline
		0 & & 0 & \eta \\
		\hline
		\eta & \eta & \eta & 0,1

		\end{array}
		\hspace{8ex}
		\begin{array}{c || c | c | c | c}
			\star & 1&0 &\alpha & \beta\\
		\hline
		\hline
		1& 1 &  & \alpha & \beta\\
		\hline
		0 & & 0 & \alpha& \beta \\
		\hline
		\alpha & \alpha & \alpha & 1,0&\beta \\
		\hline
		\beta & \beta & \beta &  \beta & 1,0,\alpha

		\end{array}
		\]
	
		\caption{The Jordan fusion law $\mathcal{J}(\eta)$ and the Monster fusion law $\mathcal{M}(\alpha,\beta)$}\label{table:jordanmonster}
	\end{table}

	\begin{definition} Let $\mathcal{F}$ be a fusion law and $A$ a commutative non-associative $\F$-algebra.
		\begin{enumerate}
			\item For $a\in A$, let $L_a$ denote the endomorphism of $A$ defined by multiplying elements with $a$. If $\lambda\in \F$ is an eigenvalue of $L_a$, the $\lambda$-eigenspace will be denoted by $A_\lambda(a)$.
			\item An idempotent $a\in A$ is an $\mathcal{F}$-axis if $L_a$ is semisimple, $\Spec(L_a)\subseteq \mathcal{F}$ and for all $\lambda,\mu\in \Spec(L_a)$:
			\[ A_{\lambda}(a)A_{\mu}(a) \subseteq \bigoplus_{\nu \in \lambda \star \mu} A_{\nu}(a).\]
			An $\mathcal{F}$-axis is primitive if $A_1(a) = \langle a \rangle$.
			\item $(A,X)$ is a (primitive) $\mathcal{F}$-axial algebra if $X\subset A$ is a set of (primitive) $\mathcal{F}$-axes that generate $A$. If  $|X| = k\in \mathbb{N}$, we will call $(A,X)$ a $k$-generated $\mathcal{F}$-axial algebra.
				
		\end{enumerate}
	\end{definition}
	For a fusion law we will always assume $1\in \mathcal{F}$, as axes are idempotents.
	
	We introduce connectivity of axial algebras. An axial algebra being ``connected'' can be considered a weaker notion than simplicity of algebras.
	
	\begin{definition}[\cite{JustinSergey}]
		\begin{enumerate}
			\item For two primitive axes $a,b$ in an axial algebra we can uniquely decompose $b= \varphi a \bigoplus_{\lambda \in \mathcal{F}\setminus \{1 \}} b_\lambda$, with $b_\lambda \in A_\lambda(a)$. We write $\phi_a(b)\coloneqq \varphi$. 
			\item Let $\Gamma=(V,E)$ be the directed graph with vertex set $V=X$, the axes of $A$, and a directed edge $a\to b$ if and only if $\phi_a(b)\neq 0$. We say $\Gamma$ is connected if there exists a directed path from any vertex to any other vertex. We will call $\Gamma$ the projection graph of $A$.
			\item The axial algebra $A$ will be called connected if $\Gamma$ is connected.
		\end{enumerate}
	\end{definition}
Examples of Jordan type axial algebras are so-called \emph{Matsuo algebras}, the family of algebras central in the current paper. They arise from $3$-transposition groups, which we will introduce next.
\subsection{$3$-transposition groups, Fischer spaces and Matsuo algebras}
We follow \cite{Aschbacher1987} to define $3$-transposition groups, Fischer spaces and Matsuo algebras. This class of groups was introduced by Bernd Fischer in \cite{fischer} and were classified in full by Hans Cuypers and Jonathan Hall \cite{CUYPERS} up to central quotients.

\begin{definition}
\begin{enumerate}
	\item A \emph{$3$-transposition group} $(G,D)$ is a group $G$ generated by a conjugacy class of order $2$ elements $D$ such that the product of any two elements in $D$ has order at most $3$
	\item Let $(G,D)$ be a $3$-transposition group. We will call the point-line geometry with point set $D$ and line set $\{\{a,b,b^a\} \mid a,b\in D, o(ab)=3\}$ the \emph{Fischer space} of $(G,D)$.
	\item Given a field $\F$, a $3$-transposition group $(G,D)$ and $\eta\in \F$, we write $M_{\eta}(\F,(G,D))$ for the algebra $\F D$ with multiplication
		\[ a \cdot b \begin{cases}
			a &\text{ if } a=b,\\
			0 &\text{ if } o(ab) =2,\\
			\tfrac{\eta}{2}(a+b-a^{b}) &\text{ if } o(ab)=3.
		\end{cases}\]
		When all parameters are clear from context, we also write $M(G)$ for $M_{\eta}(\F,(G,D))$.
	\item We call direct sums of the algebras above \emph{Matsuo algebras}.
\end{enumerate} 
\end{definition}

Note that some authors define $3$-transposition groups such that $D$ is a union of conjugacy classes instead of just one conjugacy class. It is an easy exercise to check that this more general definition admits only direct products of $3$-transposition groups as in the definition above.

\begin{example}\label{ex:3transpo}
\leavevmode
	\begin{enumerate}
		\item Let $S_n$ be the symmetric group on $[n] = \{1,\dots,n\}$, and let $D = \{(ij) \mid i,j\in [n], i\neq j \}$. Then it is a well-known fact that $D$ is a conjugacy class of order $2$ elements that generates the group $S_n$, and that two transpositions either commute, or their product has order $3$. The symmetric group is one of the prototypical examples of a $3$-transposition group. In this case, the corresponding Matsuo algebra is always a Jordan algebra, as shown in \cite{felixtom}. 
		\item More generally, given any simply laced Weyl group $W$ (i.e.\@ a Weyl group of an irreducible root system $\Phi$ of type $A$, $D$ or $E$), we can consider the conjugacy class of reflections $D= \{\sigma_\alpha \mid \alpha \in \Phi\}$. Then, given two reflections $\sigma_\alpha,\sigma_\beta$, either $\alpha \perp \beta$, and then the two reflections commute, or $\alpha+\beta$ is a root (without loss of generality) and $\sigma_\alpha \sigma_\beta \sigma_\alpha = \sigma_{\alpha+\beta}=\sigma_\beta \sigma_\alpha \sigma_\beta$, proving that their product has order $3$. These are slightly more general than just the symmetric groups, but behave in very similar ways. The conjugacy class $D$ of reflections in this case is in bijection with $\Phi^+$, and hence, we can identify $D$ with $\Phi^+$, as we will do in \cref{sec:4gen}.
		\item Starting from a simply laced Weyl group $W = W(\Phi)$, we can use the action of $W$ on the root lattice, which is by definition an abelian group, to find a new $3$-transposition group. Write $\Lambda =\Z\Phi$, and let $G = \Lambda/3\Lambda \rtimes W$, where the action of $W$ on $\Lambda /3\Lambda$ is the action induced by the action on $\Phi$. Then the set $D=\{(\varepsilon \alpha ,\sigma_\alpha)\mid \alpha \in \Phi, \varepsilon\in \F_3\}$ is a conjugacy class of order $2$ elements, since $(0,\sigma_\alpha)^{(\alpha,0)} = (\alpha,\sigma_\alpha)$, and $\langle (0,\sigma_\alpha)\mid \alpha \in \Phi\rangle =W$. Moreover, the product of two elements in $D$ is at most order $3$, since $(\varepsilon_1\alpha ,  \sigma_\alpha)^{(\varepsilon_2\beta ,  \sigma_\beta)} = (\varepsilon_1\alpha ,  \sigma_\alpha)$ if $\alpha\perp \beta$, $(\varepsilon_1\alpha ,  \sigma_\alpha)^{(\varepsilon_2\beta ,  \sigma_\beta)} = ((\varepsilon_1+\varepsilon_2)(\alpha+\beta) ,  \sigma_{\alpha+\beta}) = (\varepsilon_2\beta ,  \sigma_\beta)^{(\varepsilon_1\alpha ,  \sigma_\alpha)}$ if $\alpha+\beta\in \Phi$ and $(\varepsilon_1\alpha,\sigma_\alpha)^{(\varepsilon_2\alpha,\sigma_\alpha)} = (-(\varepsilon_1+\varepsilon_2)\alpha,\sigma_\alpha) = (\varepsilon_1\alpha,\sigma_\alpha)^{(\varepsilon_2\alpha,\sigma_\alpha)}$. Letting $X_n$ be the Dynkin type of the root system $\Phi$, then we will write $W_3(\tilde{X}_n)$ for this group.
		\item The example above can be generalized to other $3$-transposition groups, as we will do for the smallest simply laced Weyl group, $C_2$. In that case, let $n$ be an integer, and $\F_3^n\rtimes C_2$ the semidirect product where $C_2$ acts on $\F_3^n$ by inverting each element. Then the set $D = \{ (v,\sigma) \mid v\in \F_3^n \}$  is a conjugacy class of order $2$ elements, in bijection with $\F_3^n$. It can be checked as in the previous example that the product of any two elements in this set has order $3$, and $(v,\sigma)^{(w,\sigma)} = (-v-w,\sigma)$. Groups with this property are called of \emph{Moufang type}, due to their connection with commutative Moufang loops, see e.g.\@ \cite[Section 3]{halltriplesystem}. Such a group will often be denoted by $3^n\colon 2$.
		\item Another, non-affine, $3$-transposition group is given by the Hall group, the first non-trivial example found by Marshall Hall. This is a $3$-transposition group $(G,D)$ generated by four transpositions such that $|D|=81$. For a nice model of this group, see \cite[Section 4]{halltriplesystem}. We will write $[3^{10}]\colon 2$ to denote this group.
	\end{enumerate}
\end{example}

It is easy to check that Matsuo algebras $(M(G),D)$ with parameter $\eta$ are primitive axial algebras with respect to the fusion law $\mathcal{J}(\eta)$ see, e.g.\@ \cite[Theorem~6.4]{axialjordan}.  It is also clear that connected Matsuo algebras are precisely the ones that come from exactly one $3$-transposition group, instead of being direct sums of several Matsuo algebras.

In the classification of $3$-transposition groups, a local approach is often fruitful, and one of the first building blocks of the theory is the classification of $4$-generated groups, a well-known result originally proved by Zara in his thesis, but independently proven by various authors.

\begin{lemma}\label{lem:3gen}
	Let $(G,D)$ be a $3$-transposition group generated by three transpositions and no less. Then either $G \cong S_4$ or $G \cong \SU_3(2)'/Z$, where $Z$ is a central subgroup of $\SU_3(2)'$.
\end{lemma}
\begin{proof}
	This is \cite[(1.6)]{fischer}, reformulated as in \cite[Proposition (4.1)]{hallsoicher}.
\end{proof}

Note that the $3$-transposition group $3^2\colon 2$ from \cref{ex:3transpo}(iv) is in fact isomorphic to a central quotient of $ \SU_3(2)'$. Hence we can think of $3$-generated $3$-transposition groups as either origin	ating from $S_4$ or $3^2\colon 2$. 

\begin{proposition}\label{prop:4gentranspo}
	If $(G,D)$ is a $3$-transposition group generated by $4$ elements of $D$ but no less, then $(G,D)$ is of central type
	\begin{enumerate}
		\item $S_5$, 
		\item $W(D_4)$,
		\item $W_3(\tilde{A}_3)$, 
		\item $2^{1+6}:\SU_3(2)^\prime$,
		\item $3^3\colon 2$,
		\item  or the Hall group, denoted by $[3^{10}]\colon 2$. 
	\end{enumerate}
\end{proposition}
\begin{proof}
	See \cite[Proposition 4.2]{hallsoicher}.
\end{proof}
However, two $4$-generated $3$-transposition groups can be strictly contained within each other, as we will see in the next lemma. We will call a subgroup $H\leq G$ of a $3$-transposition group $(G,D)$ a $D$-subgroup if $H=\langle H\cap D\rangle$. 

\begin{lemma}\label{lem:contained4gen}\leavevmode
\begin{enumerate}
	\item Any $2$-generated $D$-subgroup contained in a $3$-transposition group $(G,D)$ of central type $2^{1+6}:\SU_3(2)^\prime$ is contained in a $4$-generated $D$-subgroup of central type $W(D_4)$.
	\item Any $2$-generated $D$-subgroup contained in a $3$-transposition group $(G,D)$ of central type $[3^{10}]\colon 2$ is contained in a $4$-generated $D$-subgroup of central type $3^3\colon 2$.
\end{enumerate}
\end{lemma}
\begin{proof}
	Item (i) can be deduced from \cite[Proposition 5.2]{hallsoicher} and (ii) from \cite[Proposition 7.2]{hallsoicher}.
\end{proof}

Given a $3$-transposition group $(G,D)$, we can associate a partial linear space $\mathcal{G}_D$ to it by taking the set of involutions $D$ as points, and taking as lines the subsets $\{a,b,b^{a}\}$ whenever $a$ and $b$ do not commute, equivalently, when $o(ab)=3$. This will be a so-called \emph{Fischer space}. These are partial linear spaces defined by their $3$-generated subgeometries.
\begin{definition}\label{def:fischer}
	A \emph{Fischer space} is a partial linear space $(\mathcal{P},\mathcal{L})$ such that the subspace spanned by two intersecting lines is either a dual affine plane of order $2$, or an affine plane of order $3$ (see \cref{fig:3gen}).
\end{definition}

In  \cref{fig:3gen}, we have explicitly labeled the nodes of the Fischer space by the transpositions that correspond to it. In this case, we identified the set of involutions in $3^2\colon 2$ with ${\mathbb{F}_3}^2$ for the affine plane. 
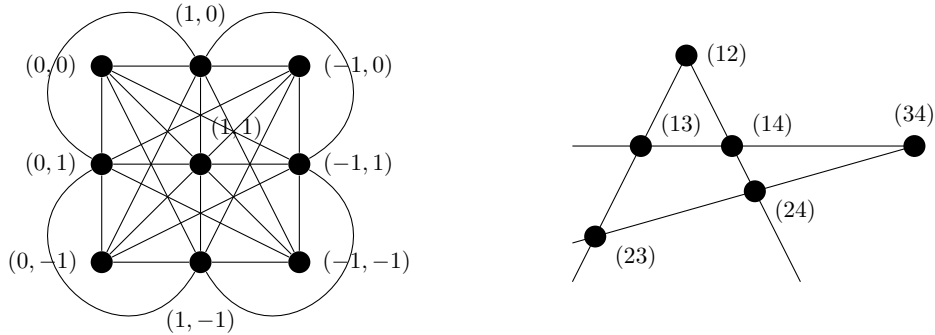
\begin{figure}[h]
\begin{center}
\begin{minipage}{0.48\textwidth}
\centering
	\begin{tikzpicture}[scale=.65,every node/.style={scale=.8,fill=black,circle}]
        \AThreeDrawing{$(0,0)$}{$(1,0)$}{$(-1,0)$}{$(0,1)$}{$(1,1)$}{$(-1,1)$}{$(0,-1)$}{$(1,-1)$}{$(-1,-1)$}
\end{tikzpicture}
\end{minipage}
\begin{minipage}{0.48\textwidth}
\centering
	\begin{tikzpicture}[scale=.6, every node/.style={scale=.8}]
             \DATwoDrawing{$(12)$}{$(13)$}{$(23)$}{$(34)$}{$(24)$}{$(14)$}
\end{tikzpicture}
\end{minipage}
\end{center}        	
\caption{The affine plane of order $3$ and dual affine plane of order $2$.}\label{fig:3gen}
\end{figure}

We will say that two $3$-transposition groups have the same \emph{central type} if they have isomorphic (in the sense of point-line geometries) Fischer spaces. There are two types of planes (subgeometries generated by $3$ points and no less) in Fischer spaces. It turns out that Fischer spaces and direct products of $3$-transposition groups are equivalent notions, where connected Fischer spaces correspond to $3$-transposition groups. This was proved by Buekenhout in unpublished notes (see \cite[Theorem~3.1]{CUYPERS} for a statement). In \cref{sec:4gen,sec:classFischerspaces}, we will use this to consider $3$-transposition groups combinatorially through their Fischer spaces. We will need one more small lemma concerning \emph{symplectic} $3$-transposition groups, i.e.\@ those $3$-transposition groups that do not contain any $3$-generated subgroups of type $3^3\colon 2$.

\begin{lemma}\label{lem:linetransitive}
	Let $(G,D)$ be a $3$-transposition group of symplectic type. Then $G$ acts transitively on pairs of noncommuting points of $D$.
\end{lemma}
\begin{proof}
	Let $(x,y)$,$(u,v)$ be two pairs of noncommuting points. Since $D$ is a conjugacy class in $G$, we can assume $x=u$. It remains to check that there exists an automorphism fixing $x$ and sending $y$ to $v$. 
	Then the Fischer space spanned by $x,y$ and $v$ is isomorphic to the Fischer space $S_3$ or $S_4$. In the first case, conjugation by $x$ switches $y$ and $v$. In the second case, there is a third point $z$ that commutes with $x$. Then conjugation by $zx$ switches $y$ and $v$.
\end{proof}

\section{Planar relations for derivations}\label{sec:planarrel}
In this section, let $M$ be a Matsuo algebra with parameter $\eta=\frac{1}{2}$ and $D$ the canonical basis of axes corresponding to the set of transpositions in the $3$-transposition group. Let $d\colon M\to M$ be any linear map. We will write $d(a) = \sum_{b\in D} d(a)_bb$. So, $d(a)_b$ is the coefficient of $d(a)$ at $b\in D$ with respect to the basis $D$.

\begin{lemma}\label{lem:localrels}
	Let $M(G,D)$ be a Matsuo algebra with parameter $\eta=1/2$. A linear map $d\colon M \to M$ is a derivation of $M$ if and only if it satisfies the following relations, with $a,b\in D$:
	\begin{enumerate}[label = (R\arabic*)]
		\item $d(a)_a =0$,
		\item $d(a)_b = -d(a)_{b^a}$ for $a\not\perp b$,
		\item $d(a)_b=0$ for $a\perp b$,
		\item $d(a)_c + d(b)_c +d(a)_{c^{ab}} +d(b)_{c^{ab}} =0$ for $a\perp b$ and $c$ a common neighbour of $a,b$,
		\item $d(a^b)_e = d(a)_e + d(b)_{e^a}$ for $a\not\perp b,e$ and $b\perp e$,
		\item $d(a^b)_e = d(a)_{e^b} +d(b)_{e^a}$ for $a\not\perp b$, $e\not\perp a,b,a^b$,
		\item $2d(b)_a+ d(a)_b + d(a^b)_a - \sum_{\substack{a\perp e \\ b \not\perp e}} d(b)_e$ for $a\not\perp b$.
	\end{enumerate}
\end{lemma}
\begin{proof}
	This boils down to a straightforward computation. Note that $d$ is a derivation if and only if $d(a)b+ad(b) = d(ab)$ for all axes $a,b\in D$, since $D$ is a basis for $M$. We remark that $ad(a) = \frac{1}{2}d(a)$ for $a\in D$. This is equivalent to (R1)--(R3).
	
	Now, if $a,b\in D$ are orthogonal, we have $d(a)b+ad(b) = 0$. Then we  get 
	\begin{align*}
		d(a)b+ad(b) &= \sum_{\{c,c^b\}\not\perp a,b} d(b)_c a\cdot(c-c^b) + \sum_{\{c,c^a\}\not\perp a,b} d(a)_c b\cdot(c-c^a) \\
		&= \frac{1}{4}\sum_{\{c,c^b\}\not\perp a,b} d(b)_c (c-c^a-c^b + c^{ab}) \\
		&\hspace{3ex}+ \frac{1}{4}\sum_{\{c,c^a\}\not\perp a,b} d(a)_c (c - c^b-c^a+c^{ab}) \\
		&= \frac{1}{4}\sum_{c\not\perp a,b} (d(a)_c + d(b)_c +d(a)_{c^{ab}} +d(b)_{c^{ab}})c,
	\end{align*}
	where we used (R1)--(R3) in the first step. This is equivalent with the relations (R4).
	
	The case $d(a)b+ad(b) = d(a b)$ for $a,b$ non-orthogonal is expressed by relations (R5)-(R7). Here the relations (R7) correspond to the coefficients of $a$ and $b$ in the expression    $d(a)b+ad(b)- d(a b)$.
\end{proof}

The reader is invited to draw pictures of the configurations occuring in each of these relations. Each of the relations (R1)--(R6) occur in a plane of the Fischer space, since the diameter of a connected Fischer space is $2$. For (R2)--(R5) the plane is symplectic, for (R6) it is affine. Hence, whether a linear map satisfies relations (R1)--(R6) can be checked entirely locally. 
This perspective will allow us to classify derivations\index{derivations!of an algebra} in arbitrary Matsuo algebras, as long as $\kar k \neq 3$.
\section{Derivations in $4$-generated subalgebras}\label{sec:4gen}
In this section, we will prove for any Matsuo algebra $M= M(G,D)$, two collinear points $a,b\in D$ and any derivation $d\colon M \to M$, that the coefficient $d(a)_b$ is very often equal zero, depending on the $4$-generated subspaces the line $\{a,b,b^a\}$ is contained in.
We will first consider derivations for subspaces of types $W(D_4)$, $3^3\colon 2$ and $3^3:S_4$, and then use \cref{lem:contained4gen} to handle the other $4$-generated cases.

\subsection{Subspaces of type $W(D_4)$.}

In this case, let $\Phi$ be the root system of type $D_4$ with basis $\alpha,\beta,\gamma,\delta$ such that $\alpha$ corresponds to the middle node in the Dynkin diagram of $D_4$, i.e. $\alpha^\vee(\beta) = \alpha^\vee(\gamma) = \alpha^\vee(\delta)=-1$ and $\beta^\vee(\gamma) = \beta^\vee(\delta)=\gamma^\vee(\delta) = 0$. We will write $D \coloneqq \{\chi  \,| \,\chi \in \Phi^+ \}$ for the set of $3$-transpositions in $W = W(D_4)$, as in \cref{ex:3transpo}.
\begin{lemma}\label{lem:D4orth}
	Let $M = M(G,D)$ be a Matsuo algebra, and let $d\colon M\to M$ be a derivation. Let $a,b\in D$ be such that the line through $ a$ and $b$ is contained in a Fischer subspace $F$ of type $W(D_4)$. Write $c_1,c_2,c_3\in F$ for the three transpositions orthogonal to $a$ in $F$. Then $d(a)_b=d(c_1)_{a^b}+d(c_2)_{a^b}+d(c_3)_{a^b}$.
\end{lemma}
\begin{proof}
	The group $W(D_4)$ acts transitively on the lines of its Fischer space (see \cref{lem:linetransitive}), so it suffices prove this for $a=\alpha,b=\beta$ as above. An important property of the Fischer space of $W(D_4)$ is that all points orthogonal to a given point are mutually orthogonal. We will use this property without mentioning it further.
	We can  use (R4) three times to obtain
	\begin{align}
		d(\alpha)_\beta + d(\alpha)_\gamma &= -d(\beta+\gamma+\alpha)_\beta - d(\beta+\gamma+\alpha)_\gamma, \label{eq:D41}\\
		d(\alpha)_\beta + d(\alpha)_\delta &= -d(\beta+\delta+\alpha)_\beta - d(\beta+\delta+\alpha)_\delta, \label{eq:D42}\\
		-d(\alpha)_\gamma - d(\alpha)_\delta &= d(\alpha+\gamma+\delta)_\gamma+d(\alpha+\gamma+\delta)_\delta. \label{eq:D43}
	\end{align}
	Summing over \cref{eq:D41,eq:D42,eq:D43} gives 
	\begin{multline}\label{eq:D44}
		2d(\alpha)_\beta = d(\alpha+\gamma+\delta)_\gamma+d(\alpha+\gamma+\delta)_\delta\\-d(\beta+\gamma+\alpha)_\beta - d(\beta+\gamma+\alpha)_\gamma -d(\beta+\delta+\alpha)_\beta - d(\beta+\delta+\alpha)_\delta.
	\end{multline}
	Applying (R4) three times again, we get, writing $\rho = 2\alpha +\gamma + \beta +\delta$,
	\begin{align}
		-d(\alpha+\beta+\gamma)_\beta-d(\alpha+\beta+\delta)_\beta & \nonumber\\&\hspace{-5ex}=  d(\alpha+\beta+\gamma)_\rho+d(\alpha+\beta+\delta)_\rho, \label{eq:D45}\\
		d(\alpha+\beta+\gamma)_\rho+d(\alpha+\gamma+\delta)_\gamma & \nonumber\\&\hspace{-5ex}= -d(\alpha+\beta+\gamma)_\gamma-d(\alpha+\gamma+\delta)_\rho, \label{eq:D46} \\
		d(\alpha+\beta+\delta)_\rho+d(\alpha+\gamma+\delta)_\delta & \nonumber\\&\hspace{-5ex}= -d(\alpha+\beta+\delta)_\delta-d(\alpha+\gamma+\delta)_\rho. \label{eq:D47}
	\end{align}
	Substituting \cref{eq:D45,eq:D46,eq:D47} in \cref{eq:D44} gives
	\begin{multline*}
		2d(\alpha)_\beta = -d(\alpha+\beta+\delta)_\delta-d(\alpha+\gamma+\delta)_\rho \\-d(\alpha+\beta+\gamma)_\gamma-d(\alpha+\gamma+\delta)_\rho- d(\alpha+\beta+\gamma)_\gamma- d(\alpha+\beta+\delta)_\delta.
	\end{multline*} 
	Dividing both sides by $2$ and using (R2) on all terms on the right hand side shows the lemma.
\end{proof}
\begin{lemma}\label{lem:planarD4}
	Under the same assumptions as in \cref{lem:D4orth}, we have that $d(a)_b=0$.
\end{lemma}
\begin{proof}
	The group $W(D_4)$ acts transitively on lines of the Fischer space of type $W(D_4)$. Thus, it suffices to prove $d(\alpha)_\beta =0$, where $\alpha,\beta$ are as above. By \cref{lem:D4orth}, we have
	\begin{equation}\label{eq:D4start}
		d(\alpha)_\beta=d(\alpha+\beta+\gamma)_{\alpha+\beta}+d(\alpha+\beta+\delta)_{\alpha+\beta}+d(\alpha+\gamma+\delta)_{\alpha+\beta}.
	\end{equation}
	We can apply (R5) three times to get
	\begin{align}
		d(\alpha+\beta+\gamma)_{\alpha+\beta} &= d(\beta)_{\alpha+\beta} + d(\gamma+\alpha)_\alpha, \label{eq:D48}\\
		d(\alpha+\beta+\delta)_{\alpha+\beta} &= d(\beta)_{\alpha+\beta}+d(\delta+\alpha)_\alpha, \label{eq:D49}\\
		d(\alpha+\gamma+\delta)_{\alpha+\beta}&= d(\beta)_{\alpha+\beta}+d(\alpha+\beta+\gamma+\delta)_\alpha. \label{eq:D410}
	\end{align}
	By substituting \cref{eq:D48,eq:D49,eq:D410} in \cref{eq:D4start}, and making use of \cref{lem:D4orth} again we get
	\begin{equation}\label{eq:D4endpart1}
		d(\alpha)_\beta = 3d(\beta)_{\alpha+\beta} + d(\alpha+\beta)_{\beta}.
	\end{equation}
By using the symmetry of the Fischer space (i.e. the transitivity of the Weyl group), we also have
	\begin{equation}\label{eq:D4endpart2}
		d(\beta)_\alpha = 3d(\alpha)_{\alpha+\beta} + d(\alpha+\beta)_{\alpha}.
	\end{equation}
	Summing \cref{eq:D4endpart1,eq:D4endpart2}, we obtain $d(\alpha)_\beta + d(\beta)_\alpha = 3(d(\alpha)_{\alpha+\beta}+d(\beta)_{\alpha+\beta})=-3(d(\alpha)_{\beta}+d(\beta)_{\alpha})$, hence $d(\alpha)_{\beta} = -d(\beta)_{\alpha}$. By transitivity of the Weyl group we get $d(\alpha+\beta)_\beta = -d(\beta)_{\alpha+\beta}$. Plugging these two equations into \cref{eq:D4endpart1}, we finally obtain $d(\alpha)_\beta = -d(\beta)_\alpha=0$.
\end{proof}

\subsection{Subspaces of type $3^3\colon 2$.}

In this case, we can identify the Fischer space of type $3^3\colon 2$ with $\mathbb{F}_3^3$, where lines are of the form $\{x,y,-x-y\}$, see \cref{ex:3transpo}.
\begin{lemma}\label{lem:planaraffine}
	Let $M = M(G,D)$ be a Matsuo algebra and suppose $\kar k \neq 3$, and let $d\colon M\to M$ be a derivation. Let $x,y\in D$ such that the line through $ x$ and $y$ is contained in a Fischer subspace of type $3^3\colon 2$. Then $d(x)_y=0$. 
\end{lemma}
\begin{proof}
	Identifying $x$ and $y$ with elements of $\mathbb{F}_3^3$, we will prove that $d(x)_y =0$. Suppose $z,w$ are two more transpositions such that $x,y,z,w$ generate the Fischer space of $3^3\colon 2$.  
	We can apply (R6) five times to get the following equations
	\begin{align*}
		d(x)_y &= d(z)_{-y+z+x} +d(-x-z)_{-y-z},\\
		d(z)_{-y+z+x} &= d(z+w-x)_{z+x+y+w} +d(z+x-w)_{y+z-w},\\
		d(-x-z)_{-y-z} &= d(w+x-z)_{w+y-z} +d(-z-w)_{-z-w+y-x}, \\
		d(x)_{-x-y} &= d(z+w-x)_{z+x+y+w} +d(-z-w)_{y-x-w-z}, \\
		d(x)_{-x-y} &= d(w+x-z)_{w+y-z} +d(z+x-w)_{y+z-w}.
	\end{align*}
	This proves $3d(x)_y =0$, since $d(x)_{-x-y} = -d(x)_y$ by (R2). The reader is invited to actually draw the Fischer space to visualize the equations.\end{proof}
\subsection{Subspaces of type $3^3:S_4$.}
We can identify the Fischer space of type $3^3:S_4$ with $D= \{(0,\sigma_\chi),(\pm \chi , \sigma_\chi) \mid \chi \in \Phi^+\}$ where $\Phi$ is a root system of type $A_3$ with basis $\alpha,\beta,\gamma$, where $\alpha\not\perp \beta,\gamma$, as in \cref{ex:3transpo}.  It is immediately clear that we can distinguish the lines of the Fischer space into two different types.
\begin{lemma}\label{lem:lineorbits}
	The lines of the Fischer space\index{Fischer space} of type $3^3:S_4$ decompose into two orbits under the action of $G= 3^3:S_4$, the lines with representative $\{(0,\sigma_\alpha),(\alpha,\sigma_\alpha),(-\alpha,\sigma_\alpha)\}$, which we will call \emph{vertical lines} and lines with representative $\{(0,\alpha),(0,\beta),(0,\alpha+\beta) \}$, which we will call \emph{horizontal lines}.
\end{lemma}
\begin{proof}
	Recall from \cref{ex:3transpo}(iii) that  $(\varepsilon_1\alpha ,  \sigma_\alpha)^{(\varepsilon_2\beta ,  \sigma_\beta)} = ((\varepsilon_1+\varepsilon_2)(\alpha+\beta) ,  \sigma_{\alpha+\beta})$ and $(\varepsilon_1\alpha,\sigma_\alpha)^{(\varepsilon_2\alpha,\sigma_\alpha)} = (-(\varepsilon_1+\varepsilon_2)\alpha,\sigma_\alpha)$. Lines in the Fischer space of type $3^3\colon S_4$ are either of the form \[L_1=\{(\varepsilon_\alpha\alpha,\sigma_\alpha),(\varepsilon_\beta\beta,\sigma_\beta),((\varepsilon_\alpha+\varepsilon_\beta)(\alpha+\beta),\sigma_{\alpha+\beta})\}\] or of the form \[L_2 = \{(\varepsilon \alpha,\sigma_\alpha)\mid \varepsilon\in \F_3\}.\] It is clear that applying an element of $D$ to a line of the second form always gives another line of the second form, and this action is in fact transitive because the Weyl group acts transitively on the root system. Lines of the first form can be shown to be precisely one orbit under the action of $3^3:S_4$, by \cref{lem:linetransitive}. Indeed, we can first apply $(-\varepsilon_\alpha \alpha,\sigma_\alpha)$ to $L_1$ to suppose $\varepsilon_\alpha =0$. Then, take $\gamma \in \Phi^+$ with $\gamma \not\perp \alpha$ and $\gamma \neq \beta$. Then $L_1$ and $(0,\gamma)$ together generate a dual affine plane, hence $L_1$ can be mapped to $\{(0,\alpha),(0,\gamma),(0,\sigma_\alpha(\gamma)\}$ by \cref{lem:linetransitive}. Clearly, this line can be sent to any other line in the ``zero'' dual affine plane consisting of the points $(0,\alpha)$, $\alpha \in \Phi^+$, again by \cref{lem:linetransitive}. 
\end{proof}
In this subsection, we prove, given a Matsuo algebra $M = M(G,D)$, a derivation $d\in \Der (M)$ and two collinear points $a,b\in D$, that   $d(a)_b=0$ whenever the line $L$ spanned by $a,b$ is contained in a Fischer space $F$ of type $3^3\colon S_4$ and $L$ is horizontal in $F$. In what follows, given $\alpha\in \Phi^+$, we will write $(0,\alpha),(+,\alpha)$ and $(-,\alpha)$ for $(0,\sigma_\alpha),(\alpha,\sigma_\alpha)$ and $(-\alpha,\sigma_\alpha)$ respectively.
\begin{lemma}\label{lem:3S4vert}
Let $M = M(G,D)$ be a Matsuo algebra, and let $d\colon M\to M$ be a derivation. Let $a,b_1\in D$ such that the line through $ a$ and $b_1$ is a horizontal line in a Fischer subspace $F$ of type $3^3\colon S_4$. Write $\{b_1,b_2,b_3\}$ for the unique vertical line through $b_1$ in $F$. Then we have
	\[ d(a)_{b_1}+d(a)_{b_2}+d(a)_{b_3}=0. \]
\end{lemma}
\begin{proof}
	We can assume without loss of generality that $a= (0,\alpha),b=(0,\beta)$ by \cref{lem:lineorbits}.
	Applying (R5), we have
	\begin{equation*}
		d(0,\alpha)_{(0,\beta)} =d(0,\alpha+\gamma)_{(0,\beta)}-d(0,\gamma)_{(0,\alpha+\beta)}.
	\end{equation*}
	By applying (R6) twice on the right hand side, we get, with $\rho = \alpha +\beta +\gamma$,
	\begin{multline*}
		d(0,\alpha)_{(0,\beta)} \\= d(+,\alpha+\gamma)_{(-,\rho)} +d(-,\alpha+\gamma)_{(+,\rho)} -d(+,\gamma)_{(-,\rho)}-d(-,\gamma)_{(+,\rho)}.
	\end{multline*}
	Applying (R5) twice we finally obtain
	\begin{equation*}
		d(0,\alpha)_{(0,\beta)} = -d(0,\alpha)_{(+,\beta)}-d(0,\alpha)_{(-,\beta)}.\qedhere
	\end{equation*}
\end{proof}

\begin{lemma}\label{lem:planar3S4}
	Let $M = M(G,D)$ be a Matsuo algebra, suppose $\kar k \neq 3$, and let $d\colon M\to M$ be a derivation. Let $a,b\in D$ such that the line through $ a$ and $b$ is a horizontal line in a Fischer subspace $F$ of type $3^3\colon S_4$. Then $d(a)_b=0$. 
\end{lemma}
\begin{proof}
We can again assume without loss of generality that $a= (0,\alpha),b=(0,\beta)$ by \cref{lem:lineorbits}.
	Applying (R5) three times, we get
	\begin{align}
		d(0,\alpha)_{(0,\beta)} &= d(0,\alpha+\gamma)_{(0,\beta)} - d(0,\gamma)_{(0,\alpha+\beta)},\label{eq:3S41}\\
		d(0,\alpha)_{(0,\beta)} &= d(+,\alpha+\gamma)_{(0,\beta)} - d(+,\gamma)_{(0,\alpha+\beta)}, \label{eq:3S42}\\
		d(0,\alpha)_{(0,\beta)} &= d(-,\alpha+\gamma)_{(0,\beta)} - d(-,\gamma)_{(0,\alpha+\beta)}. \label{eq:3S43}
	\end{align}
	Applying (R4) six times, we get
	\begin{align}
		d(0,\gamma)_{(0,\alpha+\beta)} &= d(0,\gamma)_{(0,\alpha)}-d(0,\beta)_{(0,\alpha+\beta)}-d(0,\beta)_{(0,\alpha+\gamma)}, \label{eq:3S44}\\
		d(+,\gamma)_{(0,\alpha+\beta)} &= d(+,\gamma)_{(+,\alpha)}-d(-,\beta)_{(0,\alpha+\beta)}-d(-,\beta)_{(-,\alpha+\gamma)},\label{eq:3S45} \\
		d(-,\gamma)_{(0,\alpha+\beta)} &= d(-,\gamma)_{(-,\alpha)}-d(+,\beta)_{(0,\alpha+\beta)}-d(+,\beta)_{(+,\alpha+\gamma)}, \label{eq:3S46} \\
		d(0,\alpha+\gamma)_{(0,\beta)} & \nonumber\\&\hspace{-6ex}= d(0,\alpha+\gamma)_{(0,\alpha)} -d(0,\alpha+\beta)_{(0,\beta)}-d(0,\alpha+\beta)_{(0,\gamma)}, \label{eq:3S47} \\
		d(+,\alpha+\gamma)_{(0,\beta)} & \nonumber\\&\hspace{-6ex}= d(+,\alpha+\gamma)_{(-,\alpha)} -d(-,\alpha+\beta)_{(0,\beta)}-d(-,\alpha+\beta)_{(-,\gamma)}, \label{eq:3S48}\\
		d(-,\alpha+\gamma)_{(0,\beta)}& \nonumber\\&\hspace{-6ex}= d(-,\alpha+\gamma)_{(+,\alpha)} -d(+,\alpha+\beta)_{(0,\beta)}-d(+,\alpha+\beta)_{(+,\gamma)}. \label{eq:3S49}
	\end{align}
	By summing over \cref{eq:3S41,eq:3S42,eq:3S43} and substituting this into \cref{eq:3S44,eq:3S45,eq:3S46,eq:3S47,eq:3S48,eq:3S49} into this sum, we get
	\begin{multline}\label{eq:3S410}
		3d(0,\alpha)_{(0,\beta)} =	d(0,\beta)_{(0,\alpha+\beta)} + d(-,\beta)_{(0,\alpha+\beta)} + d(+,\beta)_{(0,\alpha+\beta)} \\
									-d(0,\alpha+\beta)_{(0,\beta)}-d(-,\alpha+\beta)_{(0,\beta)}-d(+,\alpha+\beta)_{(0,\beta)}\\
									-d(0,\gamma)_{(0,\alpha)}- d(+,\gamma)_{(+,\alpha)}-d(-,\gamma)_{(-,\alpha)}\\
									+d(0,\beta)_{(0,\alpha+\gamma)}+	d(-,\beta)_{(-,\alpha+\gamma)}+d(+,\beta)_{(+,\alpha+\gamma)}\\
									+d(0,\alpha+\gamma)_{(0,\alpha)}+d(+,\alpha+\gamma)_{(-,\alpha)}+d(-,\alpha+\gamma)_{(+,\alpha)}\\
									-d(0,\alpha+\beta)_{(0,\gamma)}-d(-,\alpha+\beta)_{(-,\gamma)}-d(+,\alpha+\beta)_{(+,\gamma)}.
	\end{multline}
	We can use (R6) four times on the right hand side of \cref{eq:3S410} to see that each of the last four lines is zero, and we get
	\begin{multline}\label{eq:3S411}
		3d(0,\alpha)_{(0,\beta)} =	d(0,\beta)_{(0,\alpha+\beta)} + d(-,\beta)_{(0,\alpha+\beta)} + d(+,\beta)_{(0,\alpha+\beta)} \\
									-d(0,\alpha+\beta)_{(0,\beta)}-d(-,\alpha+\beta)_{(0,\beta)}-d(+,\alpha+\beta)_{(0,\beta)}.
	\end{multline}
	We also have, again by (R6) that
	\begin{align*}
		d(0,\alpha)_{(-,\alpha+\beta)} &= d(+,\beta)_{(0,\alpha+\beta)} - d(+,\alpha+\beta)_{(0,\beta)}, \\
		d(0,\alpha)_{(+,\alpha+\beta)} &= d(-,\beta)_{(0,\alpha+\beta)} - d(-,\alpha+\beta)_{(0,\beta)},
	\end{align*}
	and
	\begin{align*}
		d(0,\alpha)_{(-,\alpha+\beta)} &= d(0,\beta)_{(+,\alpha+\beta)} - d(0,\alpha+\beta)_{(+,\beta)}, \\
		d(0,\alpha)_{(+,\alpha+\beta)} &= d(0,\beta)_{(-,\alpha+\beta)} - d(0,\alpha+\beta)_{(-,\beta)}. 
	\end{align*}
	Substituting this into \cref{eq:3S411} and using \cref{lem:3S4vert} two times, we get that $3d(0,\alpha)_{(0,\beta)} =0$.
\end{proof}

\subsection{Subspaces of type $2^6\colon\SU_3(2)^\prime$ and $[3^{10}]\colon 2$.}
The two most complicated $4$-generated cases are harder to understand geometrically. However, it turns out that these bigger $4$-generated $3$-transposition groups contain many other, smaller $3$-transposition groups of central type either $3^3:2$ or $W(D_4)$. Since derivations give many ``local relations'', this will be sufficient for our purposes.
\begin{lemma}\label{lem:planarothers}
   Let $M = M(G,D)$ be a Matsuo algebra, suppose that $\kar k \neq 3$, and let $d\colon M\to M$ be a derivation. Let $a,b\in D$ be collinear such that the line through $ a$ and $b$ is contained in a Fischer space of type $2^6\colon\SU_3(2)^\prime$ or $[3^{10}]\colon 2$. Then $d(a)_b=0$.
\end{lemma}
\begin{proof}
	 In case $a,b$ is contained in a Fischer space of type $2^6\colon\SU_3(2)^\prime$, then the line spanned by $a,b$ is contained in a Fischer space of type $W(D_4)$, by \cref{lem:contained4gen}. By \cref{lem:planarD4} then, $d(a)_b=0$. In case $a,b$ is contained in a Fischer space of type $[3^{10}]\colon 2$, then the line spanned by $a,b$ is contained in a Fischer space of type $3^3\colon 2$, by \cref{lem:contained4gen}. By \cref{lem:planaraffine} then, $d(a)_b=0$. This proves the lemma.
\end{proof}
We conclude this section with the following summary.
\begin{lemma}\label{lem:crux}
	Let $(G,D)$ be a $3$-transposition group, and $M$ its associated Matsuo algebra over a field $k$ not of characteristic $3$. Let $a,b\in D$ and $d\colon M \to M$ a derivation with $d(a)_b$ non-zero. Then $a,b$ are collinear, say on a line $L$. Moreover, every $4$-generated Fischer subspace $F$ that contains the line $L$ is either of type $3^3:S_4$, in which case $L$ is vertical in $F$, or of type $S_5$, or it is $3$-generated.
\end{lemma}
\begin{proof}
	This theorem immediately follows from the relation (R2) and \cref{lem:planarD4,lem:planaraffine,lem:planar3S4,lem:planarothers}.
\end{proof}
We will now prove that lines satisfying \cref{lem:crux} only exist in Matsuo algebras of type $S_n$
 (which are Jordan) or of type $3^n:W$, where $W$ is a simply laced Weyl group of rank $n$. We will call these lines ``near-solid''.
 
 \begin{definition}\label{def:near-solid}
 	A line $L=\{a,b,b^a\}$ in the Fischer space $\mathcal{G}_D$ will be called \emph{near-solid} if every $4$-generated Fischer space $\mathcal{G}^\prime$ that is not $3$-generated containing $L$ is either of central type $S_5$ or of central type $3^3\colon S_4$ and $L$ is a vertical line in $\mathcal{G}^\prime$.
 \end{definition}
 
  It is easy to show that solid lines in the sense of \cite{gorshkov2024solid} are always near-solid, by applying \cite[Theorem~1.2]{solidisjordan}. It turns out, from the classification result in this chapter, that any near-solid line is in fact solid. This is not at all clear from the definition, and it is in particular not clear why relation (R7) is somehow redundant when $\kar k\neq 3$. In fact, in characteristic $3$, the situation is vastly different. There are more Matsuo algebras with non-zero derivations in that case, and the relation (R7) does seem to play a bigger role. 

\section{Classification of near-solid lines in Matsuo algebras}\label{sec:classFischerspaces}
We will use Fischer spaces and Matsuo algebras interchangably in this section. In this section, we will classify (using the classification results of Cuypers and Hall \cite{CUYPERS}) all Fischer spaces with near-solid lines. Equivalently, we classify those Matsuo algebras which may admit non-zero derivations, in view of \cref{lem:crux}.
\begin{lemma}\label{lem:affineordual}
	A near-solid line in a Fischer space is either only contained in dual affine planes or in affine planes.
\end{lemma}
\begin{proof}
	Suppose a line $L$ is contained in both an affine plane $A$ and a dual affine plane $D$. Then the subspace generated by $A$ and $D$ is $4$-generated, and $L$ is contained in both $A$ and $D$ in this subspace. This is in contradiction with \cref{lem:crux}.
\end{proof}
In the next proposition, $\mathrm{FSym}_\Omega$ is the group of all finite permutations of some set $\Omega$.
\begin{proposition}\label{prop:sympsolidline}
	If a connected Fischer space $\mathcal{G}_D$ not generated by three elements contains a near-solid line that is only contained in dual affine planes, then the central type of $\mathcal{G}_D$ is $ \mathrm{FSym}_\Omega$ and all lines are near-solid.
\end{proposition}
\begin{proof}
	Suppose there is a near-solid line $L$ only contained in dual affine planes. Take an axis $p\in D$ on $L$. Suppose there exists an affine plane $\pi$ of order $3$ containing $p$. Then the subspace generated by $\pi,L$ is $4$-generated, containing $L$. But $L$ cannot be near-solid in this subspace $B$, by \cref{def:near-solid}.
	
	Thus $p$ can only be contained in dual affine planes of order $2$. The Fischer space is connected, so the associated group $G$ acts transitively on the point set $D$. This means any $3$-generated subspace has to correspond to a dual affine plane, and $M$ is symplectic. Now, if a Fischer space is symplectic, $G$ acts transitively on lines by \cref{lem:linetransitive}. So if one line is near-solid, in fact all lines are since the Miyamoto automorphisms act transitively on the $2$-generated subspaces of $M$.
	
	In fact, all $4$-generated connected subgroups of $G$ have to be (contained in a subgroup) of central type $S_5$. It then follows from \cite[Theorem~6.1]{hall1993} that the Fischer is the Fischer space of a symmetric group. 
\end{proof}

We will now deal with the other kind of near-solid lines.
\begin{proposition}\label{prop:affsolidline}
	If a connected Fischer space $\mathcal{G}_D$ not generated by $3$ elements contains a near-solid line that is only contained in affine planes, then the near-solid lines partition the point set and the corresponding group is a split extension of a simply laced Weyl group and an elementary abelian $3$-group.
\end{proposition}
\begin{proof}
	Suppose there is a near-solid line $L$ only contained in affine planes. The Fischer space is connected, thus the associated group $G$ acts transitively on $D$. This means that each point is on at least one near-solid line.
	
	Suppose a point $p\in D$ is contained in two near-solid lines $L_1$ and $L_2$. Then the subspace $B$ generated by $L_1$ and $L_2$ is 3-generated. If it were an affine plane, then any $4$-generated subspace containing $L_1$ and $L_2$ (which exists by assumption) should be of type $3^3\colon S_4$. But then either $L_1$ or $L_2$ is not near-solid, since near-solid lines do not intersect in Fischer spaces of type $3^3\colon S_4$. So $B$ should be a dual affine plane. But then $L_1$ can only be contained in dual affine planes by \cref{lem:affineordual}. But then our original near-solid line $L$ cannot exist by \cref{prop:sympsolidline}, a contradiction. So the near-solid lines must partition the point set. 
	
	Secondly, any affine plane must contain all near-solid lines that go through its points. Indeed, suppose we have a near-solid line $L$ and an affine plane $A$ intersecting  $L$ in a point $p$. Then they span a $4$-generated subspace in which $L$ is not near-solid by \cref{def:near-solid}, a contradiction.
	
	This implies that the Fischer space is of \emph{orthogonal type} (see \cite[p. 344]{orthogonaltype} for a definition). Indeed, let $\pi$ be an affine plane, and $p\in D$ not in $\pi$. If we denote the three near-solid lines in $\pi$ by $L_1$, $L_2$ and $L_3$, then $p$ has to be collinear either all $3$ points of $L_i$ for any $i$, or none of them, because the $L_i$ are only contained in affine planes. This means property (P) from \cite[p. 344]{orthogonaltype} holds.
	
	 In fact, the hypotheses of \cite[Theorem~6.13]{CUYPERS} (see also the definition of $\mathcal{LC}$ from \cite[p.344, above Theorem 1]{orthogonaltype}) hold. Indeed, we will prove that for every affine plane $\pi$ and line $L\subseteq \pi$ which is not near-solid, there are no points $p$ in the Fischer space such that $p$ is collinear with precisely $\pi \setminus L$. If there were such a point $p$, and $L_s$ would be a near-solid line in $\pi$, then $\langle p,L_s\rangle$ would be a dual affine plane of order 2. This is in contradiction with \cref{lem:affineordual}.	
	By \cite[Theorem~6.13]{CUYPERS}, $(G,D)$ is of the same central type as one of the following groups: 
	
	\begin{enumerate}[label = (\arabic*)]
		\item $\left(\bigoplus_{i\in I} V(W) \right): W$, where $W$ is a simply Weyl group and the index set $I$ is not empty. Here $V(W)$ is the geometric representation of $W$ over $\F_3$.
		\item $\Wr(K,\Omega)$ where $K$ is a strong $\{2,3\}$-group, and $|\Omega|\geq 4$.
		\item A group of Moufang type.
	\end{enumerate}
	
	(1) actually gives near-solid lines when $|I|=1$, as we will see in \cref{prop:spreadsolid}. In case $G$ is a group of Moufang type, the Fischer space contains no dual affine planes, and as $M$  contains at least one subalgebra of type $3^3:S_4$, this is a contradiction.
	
	Let us investigate (2). A strong $\{2,3\}$-group $K$ has to be one of the following (see \cite[Example PR.2]{CUYPERS}):
	\begin{enumerate}[label  = (2\alph*)]
		\item An elementary abelian $2$-group,
		\item A group of exponent $3$,
		\item An elementary abelian $3$-group extended by a fix point free (inverting) involution,
		\item An elementary abelian $2$-group extended by a fix point free element of order $3$.
	\end{enumerate}
	
	If $K$ does not contain order $3$ elements, the Fischer space of $\Wr(K,\Omega)$ would not contain affine planes. Thus the group associated to $\mathcal{G}_D$ cannot be of the form (2a). Now suppose $\rho,\tau \in K$ are two elements of order $3$ with $\langle \rho\rangle\neq \langle \tau \rangle$. Note that for $i,j,k\in \Omega$ we have two affine planes $A_\rho,A_\tau$ intersecting only in the line $(ij),(ik),(jk)$. The affine plane $A_\rho$ is spanned by $(ij),(ik),\rho_i\rho^{-1}_j:(ij)$, and $A_\tau$ is spanned by $(ij),(ik),\tau_i\tau^{-1}_j:(ij)$. Since any affine plane contains the near-solid lines through its points, the unique near-solid line through $(ij)$ should be contained in the intersection of $A_\rho,A_\tau$. But then the line through $(ij),(ik),(jk)$ has to be near-solid. This is a contradiction, since this line is contained in a dual affine plane.
	
	So $K$ has to contain precisely $2$ elements of order $3$. This means that either $K$ is equal to the cyclic group of order $3$, in which case $\Wr(K,\Omega)$ is of type $3^{n-1}:S_n$, or $K$ is equal to $S_3$, in which case $\Wr(K,\Omega)$ is of type $3^n:W(D_n)$. But these two classes are already contained in case (1).
	
	Note that the same argument can be used to prove that if $G$ is of the form (1), we need $|I|=1$. Thus the only possibility that remains is that $M$ has central type $W_3(\tilde{D})$, where $D$ is the Dynkin diagram of a simply laced Weyl group, i.e. the examples from \cref{ex:3transpo}(iii).
%
%
%
\end{proof}

\begin{proposition}\label{prop:spreadsolid}
	The Fischer space of type $3^n:W$ where $W$ is a simply laced Weyl group of rank $n$, contains a spread of near-solid lines.
\end{proposition}
\begin{proof}
	The class of transpositions is the set $\{ (0,\sigma_\alpha), (\pm\alpha, \sigma_\alpha) \mid \alpha \in \Phi \}$, where $\Phi$ is the associated root system of the Weyl group. Note that any line is contained in at least one affine plane. This means that the only lines that can be near-solid are the ones spanned by $(0,\sigma_\alpha), (\pm\alpha, \sigma_\alpha)$ for a fixed $\alpha$. Call such a line $L_\alpha$. Any connected $4$-generated subspace through $L_\alpha$ is however isomorphic to a Fischer space of type $3^3\colon S_4$, so $L_\alpha$ is near-solid.
\end{proof}
		
We can neatly summarize the results obtained in this section in the following way, using results from \cite{felixtom}.
\begin{theorem}\label{thm:classsolidmatsuo}
	Any connected Matsuo algebra over a field $k$ with $\kar k \neq 2,3$  that contains near-solid lines is either
	\begin{enumerate}
		\item Jordan, or,
		\item of the form $M_{\frac{1}{2}}(3^n\colon W)$, where $W$ is a simply laced Weyl group\index{Weyl group} of rank $n$. In this case, every axis is contained in precisely one near-solid line, and near-solid lines correspond to vertical lines in the corresponding Fischer space.
	\end{enumerate}
	Moreover, the only finite-dimensional connected Matsuo algebras with (possibly) infinite automorphism group in characteristic different from $3$ are contained in these two families.
\end{theorem}
\begin{proof}
	Matsuo algebras associated to symmetric groups as well as $3$-generated Matsuo algebras are Jordan algebras, by \cite{felixtom}. If it is not $3$-generated, the result follows from \cref{lem:crux,prop:sympsolidline,prop:affsolidline,prop:spreadsolid}.
\end{proof}

Perhaps surprisingly, the exceptional family in this theorem is connected precisely to root systems of simply laced type. We will see how the action on the root lattice plays a role in the next section.

\section{Automorphism groups of Matsuo algebras}\label{sec:aut3nW}

Using these results, we have the following dichotomy. We have precisely two infinite families of finite-dimensional Matsuo algebras with positive-dimensional automorphism groups: $M(S_n)$ and $M(3^n:W)$, where $W$ is a simply laced Weyl group of rank $n$. We now turn to a careful study of these automorphism groups, still under the assumption $\kar k \neq 2,3$. 
\subsection{The automorphism group of $M(S_n)$.}
Recall from \cite[Example 2.2]{felixtom} that the algebra $M(S_n)$ is isomorphic to the algebra $\ZS_n$ of $n\times n$ symmetric matrices whose rows and columns sum to zero, endowed with the Jordan product $a\bullet b = \frac{1}{2}(ab + ba)$ for $a,b\in A$. We can of course embed this algebra into the full Jordan algebra $\mathcal{S}_n$ of symmetric $n \times n$ matrices.

 Recall also that the automorphism group of $\mathcal{S}_n$ is precisely $\mathbf{PGO}_{n}$, see e.g.\@ \cite[Theorem~5.47]{gradings} and \cite[\S 23.A]{involutions}.

We can identify the space $\ZS_n$ with $V= \mathrm{S}^2W$, where $W= (1,\dots,1)^\perp$ and the orthogonality relation $\perp$ is determined by the usual scalar product of row vectors. by sending $uv\in V$ to the matrix $\frac{1}{2}(uv^\top + vu^\top)$. It is an easy exercise to check that $V$ is spanned by elements of the form $uu$ with $u \in (1,\dots,1)^\perp$. Setting $f(u,v) = u^\top v$ for all $u,v\in W$. Then $\ZS_n$ can be identified with the Jordan algebra of operators on $W$ symmetric with respect to the nondegenerate bilinear form $f$.

Using this identification, we can see the following.

\begin{theorem}
	The automorphism group of $M(S_n)$ is equal to $\mathbf{PGO}(W,f)$.
\end{theorem}
\begin{proof}
	This follows from the above, \cite[Theorem~5.47]{gradings}, and \cite[\S 23.A, p.347]{involutions}.
\end{proof}

\subsection{The automorphism group of $M(3^n\colon W )$.}
For the other family of Matsuo algebras admitting near-solid lines, we need a better handle on the algebra structure. Note that the dimension of the derivations is at most the dimension of the vertical lines, and in fact it is even smaller, as we will see in \cref{lem:idcomp3W}. We will first construct another model for these Matsuo algebras, which will allow us to directly construct the automorphism group.
\begin{definition}\label{def:modelB}
	Assume $3$ is a square in $k$, and write $\sqrt{3}$ for one of its square roots. Let $\Phi$ be a simply laced root system\index{root system!simply laced}, and for each $\alpha \in \Phi^+$, let $B_\alpha = k1_\alpha\oplus  k x_\alpha\oplus k y_\alpha  $ denote the Jordan algebra of the bilinear form (see \cite[p.14]{jacobson} for a definition) $b_\alpha$ such that $b_\alpha(x_\alpha, y_\alpha)=0$ and $b_\alpha(x_\alpha, x_\alpha)=b_\alpha(y_\alpha, y_\alpha)=\frac{9}{2}$. Define the algebra $B\coloneqq \bigoplus_{\alpha \in \Phi^+} B_\alpha$ such that the $B_\alpha$ are subalgebras. To define the multiplication between $B_\alpha$ and $B_\beta$, we write $\theta \colon B_\alpha \to B_\alpha $ for the linear map with images $\theta(1_\alpha ) = 1_\alpha, \theta(x_\alpha) = \frac{1}{2}x_{\alpha} - \frac{\sqrt{3}}{2} y_{\alpha}, \theta(y_\alpha) = \frac{1}{2}y_{\alpha}  + \frac{\sqrt{3}}{2} x_{\alpha}$ for all $\alpha\in \Phi^+$. Then for $\alpha\not\perp\beta\in \Phi^+$ and $v = \lambda 1_\beta + v_0\in B_\beta = k1_\beta \oplus (k x_\alpha\oplus k y_\alpha )$  we set
	\begin{align*}
		1_\alpha \cdot v &= \frac{\lambda}{2} 1_\alpha +\frac{1}{2}v - \frac{\lambda}{2}1_{\sigma_\alpha(\beta)},
	\end{align*}
	and for $\alpha,\beta,\alpha+\beta \in \Phi^+$ we have
	\begin{align*}
		x_\alpha \cdot x_\beta = -\frac{3}{4}\theta(y_{\alpha+\beta}) &,\;
		x_\alpha \cdot y_\beta = \frac{3}{4}\theta(x_{\alpha+\beta}), \\
		y_\alpha \cdot y_\beta = \frac{3}{4} \theta(y_{\alpha+\beta})&,\; 
		x_{\alpha+\beta}\cdot x_{\alpha} = \frac{3}{4} \theta^{-1}(y_{\beta}), \\
		x_{\alpha+\beta} \cdot y_\alpha = \frac{3}{4}\theta^{-1}(x_{\beta})&,\; 
		x_\alpha \cdot y_{\alpha+\beta} = -\frac{3}{4}\theta^{-1}(x_{\beta}),  \\
		y_\alpha \cdot y_{\alpha+\beta} = \frac{3}{4}\theta^{-1}(y_{\beta}), 
	\end{align*}
	and for $\alpha\perp \beta \in \Phi^+$ we have $B_\alpha \cdot B_\beta = 0$.
\end{definition}
\begin{remark}
	Clearly, when $k = \mathbb{R}$, the map $\theta$ models a rotation by $\frac{\pi}{3}$ of the plane spanned by $x$ and $y$.
\end{remark}
\begin{proposition}\label{prop:modelB}
	Assume $3$ is a square in $k$. 	
	Then the algebra $B$ from \cref{def:modelB} is isomorphic to the Matsuo algebra $M(3^n\colon W)$ with $W$ the Weyl group associated to the simply laced root system $\Phi$.
\end{proposition}
\begin{proof}
	Write $\sqrt{3}$ for one of the square roots of $3$ in $k$. We know that the set of $3$-tranpositions for $G= 3^n\colon W$ is of the form $D = \{  (\varepsilon\alpha, \sigma_\alpha) \mid \alpha \in \Phi,\varepsilon \in \{0,\pm1\} \}$, and $M(3^n\colon W)= k  D$ as a vector space. We can check that the map
	\begin{align*}
		B&\to M(3^n\colon W) \\
		1_\alpha &\mapsto \frac{2}{3}((0,\sigma_\alpha) + (\alpha,\sigma_\alpha) + (-\alpha, \sigma_\alpha)), \\
		x_\alpha &\mapsto \sqrt{3}((0,\sigma_\alpha) - (\alpha,\sigma_\alpha) ),\\
		y_\alpha &\mapsto 2(-\alpha, \sigma_\alpha)-(0,\sigma_\alpha) - (\alpha,\sigma_\alpha),
	\end{align*}
	is an algebra isomorphism.
\end{proof}
We identify the special orthogonal groups $\mathbf{Aut}(B_\alpha)^\circ = \mathbf{G}_\alpha \cong \mathbf{SO}(k x_\alpha \oplus k y_\alpha , b_\alpha)\cong \mathbf{SO}_2$ for all $\alpha \in \Phi^+$ by the isomorphisms sending $x_\alpha$ to $x_\beta$ and $y_\alpha$ to $y_\beta$ for two different roots $\alpha,\beta \in \Phi^+$. Now, for any $R\in k\text{-}\mathbf{Alg}$, note that $\mathbf{SO}_2(R)$ is the set of $2\times 2$ matrices $\begin{pmatrix}
	a &b \\
	c& d
\end{pmatrix}$ over $R$ satisfying the polynomial equations $ a^2+c^2 = b^2+d^2 = ad-bc =1 $ and $ab+cd=0$. Multiplying $a^2+c^2 =1$ with $d$ on both sides, we get $d= a^2d+c^2d = a(1+bc) -abc = a$. Analogously, we get $c = -b$, so  $\mathbf{SO}_2(R)=\left\{\begin{pmatrix}
	c & -s \\
	s & c
\end{pmatrix}\in M_2(R)\middle| c^2+s^2 =1 \right\}$. Using this identification, we get the following proposition.
\begin{proposition}\label{prop:automorphisms}
	Let $(\rho_\alpha)_{\alpha \in \Phi^+}$ be a family of automorphisms in the group $\mathbf{SO}_2(R) \cong \mathbf{G}_\alpha(R) $ such that $\rho_\alpha\rho_\beta  = \rho_{\alpha+\beta}$ for all $\alpha,\beta, \alpha+\beta \in\Phi^+$. Setting $B_R= B\otimes_k R$, then the map $\rho \colon B_R \to B_R \colon v \mapsto \rho_\alpha(v) $ if $v\in (B_\alpha)_R$ is an automorphism of $B_R$.
\end{proposition}
\begin{proof}
	for each $\alpha\in \Phi$, we identify $\rho_\alpha$ with the matrix $\begin{pmatrix}
		c_\alpha & -s_\alpha \\
		s_\alpha & c_\alpha
	\end{pmatrix}$ with respect to the basis ${x_\alpha,y_\alpha}$. Take $\alpha \not \perp \beta \in \Phi^+$ and $v= \lambda 1_\beta + v_0\in B_\beta$. We need to verify that
	\begin{align*}
		1_\alpha \cdot \rho(v) &= \frac{\lambda}{2}1_\alpha +\frac{1}{2}\rho(v) - \frac{\lambda}{2}1_{\gamma}
	\end{align*}
	but this is clearly true, since $\rho(v)\in B_\beta$ and $\rho$ fixes $1_\beta$. Next, assume $\alpha,\beta, \alpha+\beta=\gamma \in \Phi^+$. Then we need to verify the following equations:
	\begin{align}
		\rho_\alpha(x_\alpha) \cdot \rho_\beta(x_\beta) &= -\frac{3}{4}\rho_\gamma \circ \theta(y_{\gamma}), \label{eq:1} \\
		\rho_\alpha(x_\alpha) \cdot \rho_\beta(y_\beta) &= \frac{3}{4}\rho_\gamma \circ \theta(x_{\gamma}), \label{eq:2}\\
		\rho_\alpha(y_\alpha) \cdot \rho_\beta(y_\beta) &= \frac{3}{4} \rho_\gamma \circ \theta(y_{\gamma}), \label{eq:3}\\
		\rho_\gamma(x_{\gamma})\cdot \rho_\alpha(x_{\alpha}) &= \frac{3}{4} \rho_\beta \circ \theta^{-1}(y_{\beta}), \label{eq:4}\\
		\rho_\gamma(x_{\gamma}) \cdot \rho_\alpha(y_\alpha) &= \frac{3}{4}\rho_\beta \circ \theta^{-1}(x_{\beta}), \label{eq:5}\\
		\rho_\alpha(x_\alpha) \cdot \rho_\gamma(y_{\gamma}) &= -\frac{3}{4}\rho_\beta \circ \theta^{-1}(x_{\beta}),  \label{eq:6}\\
		\rho_\alpha(y_\alpha) \cdot \rho_\gamma(y_{\gamma}) &= \frac{3}{4}\rho_\beta \circ \theta^{-1}(y_{\beta}). \label{eq:7} 
	\end{align}
	Now note that $\theta|_{B_\delta}, \rho_\delta \in \mathbf{SO}_2(R)$ for all $\delta\in \Phi^+$, so $\rho$ commutes with $\theta^{\pm 1}$.
	Computing \cref{eq:1,eq:2,eq:3,eq:4,eq:5,eq:6,eq:7} gives
	\begin{align*}
		(c_\alpha x_\alpha + s_\alpha y_\alpha) \cdot (c_\beta x_\alpha + s_\beta y_\beta) &= -\frac{3}{4}\rho_\gamma \circ \theta(y_{\gamma}),  \\
		(c_\alpha x_\alpha + s_\alpha y_\alpha) \cdot (c_\beta y_\beta - s_\beta x_\beta) &= \frac{3}{4}\rho_\gamma \circ \theta(x_{\gamma}), \\
		(c_\alpha y_\alpha - s_\alpha x_\alpha) \cdot (c_\beta y_\beta - s_\beta x_\beta) &= \frac{3}{4} \rho_\gamma \circ \theta(y_{\gamma}), \\
		(c_\gamma x_\gamma + s_\gamma y_\gamma )\cdot (c_\alpha x_\alpha + s_\alpha y_\alpha) &= \frac{3}{4} \rho_\beta \circ \theta^{-1}(y_{\beta}), \\
		(c_\gamma x_\gamma + s_\gamma y_\gamma ) \cdot (c_\alpha y_\alpha - s_\alpha x_\alpha) &= \frac{3}{4}\rho_\beta \circ \theta^{-1}(x_{\beta}), \\
		(c_\alpha x_\alpha + s_\alpha y_\alpha) \cdot (c_\gamma y_\gamma - s_\gamma x_\gamma ) &= -\frac{3}{4}\rho_\beta \circ \theta^{-1}(x_{\beta}),  \\
		(c_\alpha y_\alpha - s_\alpha x_\alpha) \cdot (c_\gamma y_\gamma - s_\gamma x_\gamma ) &= \frac{3}{4}\rho_\beta \circ \theta^{-1}(y_{\beta}).  
	\end{align*}
	thus, using \cref{prop:modelB},
	\begin{align*}
		-\frac{3}{4}\theta( (c_\alpha c_\beta -s_\alpha s_\beta )y_\gamma -(s_\alpha c_\beta +c_\alpha s_\beta )x_\gamma)  &= -\frac{3}{4} \theta\circ \rho_\gamma(y_{\gamma}),  \\
		\frac{3}{4}\theta((c_\alpha c_\beta -s_\alpha s_\beta )x_\gamma +(s_\alpha c_\beta +c_\alpha s_\beta )y_\gamma) &= \frac{3}{4}\theta\circ \rho_\gamma(x_{\gamma}), \\
		\frac{3}{4}\theta((c_\alpha c_\beta -s_\alpha s_\beta )y_\gamma -(s_\alpha c_\beta +c_\alpha s_\beta )x_\gamma) &= \frac{3}{4} \theta\circ \rho_\gamma(y_{\gamma}), \\
		\frac{3}{4}\theta^{-1}((c_\gamma c_\alpha +s_\gamma s_\alpha )y_\beta +(c_\gamma s_\alpha -c_\alpha s_\gamma )x_\beta) &= \frac{3}{4}  \theta^{-1}\circ \rho_\beta (y_{\beta}), \\
		\frac{3}{4}\theta^{-1}((c_\gamma c_\alpha +s_\gamma s_\alpha )x_\beta +( c_\alpha s_\gamma - s_\alpha c_\gamma)y_\beta) &= \frac{3}{4}\theta^{-1}\circ \rho_\beta(x_{\beta}), \\
		-\frac{3}{4}\theta^{-1}((c_\gamma c_\alpha +s_\gamma s_\alpha )x_\beta +( c_\alpha s_\gamma - s_\alpha c_\gamma)y_\beta ) &= -\frac{3}{4}\theta^{-1}\circ \rho_\beta(x_{\beta}),  \\
		\frac{3}{4}\theta^{-1}((c_\gamma c_\alpha +s_\gamma s_\alpha )y_\beta +(c_\gamma s_\alpha -c_\alpha s_\gamma )x_\beta ) &= \frac{3}{4}\theta^{-1}\circ \rho_\beta(y_{\beta}).  
	\end{align*}
	The first three equations express that $\rho_\alpha\rho_\beta = \rho_\gamma$ while the last four express $\rho_\gamma \rho_\alpha^{-1} = \rho_\beta$.
\end{proof}
We can easily determine the dimension of the torus consisting of elements of the form above.
\begin{corollary}\label{cor:torusdim}
	Write $\alpha_1,\dots,\alpha_n$ for a basis of $\Phi$. Then for all $n$ elements $\rho_1,\rho_2,\dots,\rho_n \in \mathbf{SO}_2(R)$, there is a unique $\rho \in \mathbf{Aut}(B)$ such that $\rho|_{B_{\alpha_i}} = \rho_i$ for all $i\in \{1,\dots,n\}$. Hence, the map $\varphi \colon \mathbf{SO}_2^n \to \mathbf{Aut}(B)$ defined by sending $(\rho_1,\rho_2,\dots,\rho_n)$ to the unique $\rho \in \mathbf{Aut}(B)$ with $\rho\vert_{B_{\alpha_i}} = \rho_i$ is an embedding of algebraic groups, and $\mathbf{Aut}(B)$ contains a torus of dimension $n$.
\end{corollary}
\begin{proof}
	This follows by the previous proposition by setting $\rho_{\alpha_i+\alpha_j} = \rho_{\alpha_i}\rho_{\alpha_j}$ and proceeding inductively until the automorphism $\rho$ has been defined on each positive root $\alpha\in \Phi^+$. This process is well-defined since $\mathbf{SO}_2$ is commutative.
\end{proof}
We will set $\mathbf{T}\coloneqq \im(\varphi)$.
 We can check that $\mathbf{T}$ is the identity component of $\mathbf{Aut}(B)$.
\begin{lemma}\label{lem:idcomp3W}
	The torus $\mathbf{T}$ is the identity component of $\mathbf{Aut}(B)$ and $\mathbf{Aut}(B)$ is smooth. 
\end{lemma}
\begin{proof}
	It suffices to prove the Lie algebra of derivations of $M(3^n:W)$ is at most $n$-dimensional, since $M(3^n:W)$ and $B$ are twisted forms (\cref{prop:modelB}). If $d\colon M(3^n:W)\to M(3^n:W)$ is a derivation, we already proved $d(a)_b=0$ whenever $a$ is orthogonal to $b$ or $a,b$ span a horizontal line, by \cref{lem:planar3S4}. Let $L= \{(0,\sigma_\alpha),(\alpha,\sigma_\alpha),(-\alpha,\sigma_\alpha)\}$ be a vertical line in the Fischer space corresponding to $M(3^n:W)$. By \cref{lem:localrels}(R7), we obtain (since $L$ is not contained in any dual affine plane)
	\begin{align*}
		2d((0,\sigma_\alpha))_{(\alpha,\sigma_\alpha)} + d((\alpha,\sigma_\alpha))_{(0,\sigma_\alpha)} + d((-\alpha,\sigma_\alpha))_{(\alpha,\sigma_\alpha)}&=0,\\
		2d((\alpha,\sigma_\alpha))_{(0,\sigma_\alpha)}+d((0,\sigma_\alpha))_{(\alpha,\sigma_\alpha)} + d((-\alpha,\sigma_\alpha))_{(0,\sigma_\alpha)} &=0.\\
	\end{align*}
	Summing these two equations and dividing by $3$ gives the equation $d((0,\sigma_\alpha))_{(\alpha,\sigma_\alpha)} = -d((\alpha,\sigma_\alpha))_{(0,\sigma_\alpha)}$ for all $\alpha\in \Phi$. Moreover, the relations (R6) imply that  $d(\sigma_{\gamma})_{(\gamma, \sigma_{\gamma})} =d(\sigma_{\alpha})_{(\alpha, \sigma_{\alpha})} +d(\sigma_{\beta})_{(\beta, \sigma_{\beta})}$ for all $\alpha,\beta,\gamma= \alpha+\beta \in \Phi^+$. This implies that a derivation $d$ is completely determined by the values $d(\sigma_{\alpha})_{(\alpha, \sigma_{\alpha})}$ for all $\alpha\in \Delta$, where $\Delta$ is a basis of the root system $\Phi$. 
\end{proof}

We have found the identity component of the automorphism group, but ideally, we would like to know the full automorphism group as well. To this end, we will need to know more about the characters of $\mathbf{T}$. We will proceed in a few steps, taking inspiration from the classification of split reductive algebraic groups \cite[Chapter 21]{Milne}.
As a first step, we wish to describe the non-zero characters $\alpha$ by which the torus $\mathbf{T}$ acts on the algebra $M$. The set $P$ of these characters will turn out to be a root system. By remarking that $\mathbf{T}\unlhd \mathbf{Aut}(M)$, we can deduce there is an action on $P$ by $\mathbf{Aut}(M)$, and determine the kernel of this action to be equal to $\mathbf{T}$. After some additional arguments, we will be able to conclude that $\mathbf{Aut}(M) = \mathbf{T}\rtimes \Aut(\Phi)$, where $\Aut(\Phi)$ is regarded as a constant algebraic group.

In what follows, we will assume $k$ to be algebraically closed. This allows us to describe a mutual basis of eigenvectors. Write $\mathbf{i}\in k$ for a square root of $-1$, then we will set $e_\alpha = x_\alpha + \mathbf{i}y_\alpha$ and $f_\alpha =  x_\alpha - \mathbf{i}y_\alpha$. Recall that, since $\mathbf{Aut}(B)$ is smooth (\cref{lem:idcomp3W}), this also means we can conflate algebraic groups with their $k$-points \cite[Corollary~1.17 and Proposition~1.26]{Milne}.

Let $X(\mathbf{T}) = \{ \alpha \colon \mathbf{T} \to \mathbf{G}_m\}$ be the group of characters where for $\alpha,\beta \in X(\mathbf{T})$, the element $\alpha+\beta\in X(\mathbf{T})$ is defined to be the morphism sending $t\in \mathbf{T}$ to $\alpha(t)\beta(t)$. This way, the set $X(\mathbf{T})$ has the structure of a $\Z$-module, and $X(\mathbf{T}) \cong \Z^n$, where $0$ is the trivial character of $\mathbf{T}$, i.e.\@ $0(t) =1$ for all $t\in T$.

\begin{lemma}
	The set $\{e_\alpha,1_\alpha,f_\alpha \mid \alpha \in \Phi^+\}$ is a basis of mutual eigenvectors for $\mathbf{T}$, and the eigenspaces for $\mathbf{T}$ are given by 
		\[V_{0} = \langle 1_\alpha \mid \alpha \in \Phi^+\rangle, \]
		and
		\[V_{\alpha}  = \langle e_\alpha \rangle, V_{-\alpha}  = \langle f_\alpha \rangle \text{ for } \alpha\in \Phi^+. \]
\end{lemma}	
\begin{proof}
	By definition and \cref{prop:automorphisms}, the elements $e_\alpha,f_\alpha,1_\alpha$ are clearly eigenvectors for every element of $\mathbf{T}$. Moreover, by \cref{prop:automorphisms}, for each $\alpha,\beta \in \Phi^+$ with $\alpha\neq \beta$, it is easy to construct elements of the torus such that the set $\{e_\alpha,f_\alpha, e_\beta,f_\beta\}$ consists of eigenvectors with mutually distinct eigenvalues, all different from $1$. Indeed, we can assume $\alpha,\beta \in \Delta$ for a basis $\Delta$ of $\Phi$, and then use \cref{cor:torusdim} to construct such an element. Clearly, $\mathbf{T}$ fixes every $1_\alpha$, so this gives us the $0$-character space.  
\end{proof}
Write $\chi_{\alpha}, \chi_{-\alpha}$ for the characters corresponding to $V_{+\alpha},V_{-\alpha}$ respectively for all $\alpha \in \Phi^+$. We will prove that the $\chi_{\alpha}$ form a root system.
\begin{lemma}
	There is an isomorphism of $\Z$-modules $\iota \colon X(\mathbf{T}) \to \Z\Phi$ that sends $\chi_{\alpha}$ to $\alpha$ for each $\alpha \in \Phi$.
\end{lemma}
\begin{proof}
	First, we will prove $\chi_{\alpha_1},\dots,\chi_{\alpha_n}$ is a basis for $X(\mathbf{T})$, where $\alpha_1,\dots,\alpha_n$ is a basis for $\Phi$. Note that $\mathbf{SO}_2(k,b_\alpha) \cong \mathbf{G}_m$, by sending $\rho \in \mathbf{SO}_2(k)$ to the eigenvalue corresponding to the eigenvector $e_\alpha$ for all $\alpha\in \Phi^+$. Let $x\in X(\mathbf{T})$, then $x$ is determined by the values $x\circ \varphi(\lambda_1,\dots,\lambda_n)$ for $\lambda_1,\dots,\lambda_n\in \mathbf{G}_m$, by \cref{cor:torusdim}. Set $ x_i$ to be the exponent for which $\lambda^{x_i} = x\circ \varphi(1,\dots,1,\lambda,1,\dots,1)$ where $\lambda$ is in the $i$'th position, for $i \in \{1,\dots,n\}$ and all $\lambda \in k^\times$. Then it is clear that $x = x_1\chi_{\alpha_1}+\dots+x_n\chi_{\alpha_n}$, and by similar arguments, that $\chi_{\alpha_1},\dots,\chi_{\alpha_n}$ is linearly independent.
	
	It remains to prove that $\chi_{-\alpha} = -\chi_{\alpha}$ for all $\alpha\in \Phi^+$ and that $\chi_{\alpha+\beta} = \chi_\alpha + \chi_\beta$ for all $\alpha,\beta \in \Phi^+$ such that $\alpha+\beta\in \Phi$. Indeed, when this holds, the isomorphism $\Z\Phi \to X(\mathbf{T})$ that sends $\alpha_i$ to $\chi_{\alpha_i}$ for all $i\in \{1,\dots,n\}$ also sends each root $\alpha\in \Phi$ to $\chi_{\alpha}$, and the inverse of this isomorphism proves the lemma.
		
	But we have that for any $t\in \mathbf{T}$ and $\alpha \in \Phi^+$, we have $\chi_{\alpha}(t)\chi_{-\alpha}(t)1_\alpha = t(e_\alpha)t(f_\alpha) \linebreak= t(1_\alpha) = 1_\alpha$, proving that $\chi_{-\alpha} = - \chi_{\alpha}$. 
	
	Now let $\alpha,\beta \in \Phi^+$ such that $\alpha+\beta \in \Phi$. Then clearly $\alpha+\beta \in \Phi^+$. Using \cref{prop:modelB}, we can compute that $e_\alpha e_\beta = \mu e_{\alpha+\beta}$ for some $\mu \in k^{\times}$. Then, for any $t\in \mathbf{T}$ we get by applying $t$ to both sides that $\chi_\alpha(t) \chi_\beta(t) = \chi_{\alpha+\beta}(t)$. Since this holds for any $t\in \mathbf{T}$, we have $\chi_\alpha +\chi_\beta = \chi_{\alpha+\beta}$.
\end{proof}
The previous lemma shows that there is a natural way to view the non-zero characters occurring in $B$ as a representation for $\mathbf{T}$ as a root system, and it turns out that $\mathbf{Aut}(B)$ acts on these characters by permuting them. Indeed, if $\sigma$ is an automorphism of $B$ and $\chi \in X(\mathbf{T})$, then $t^\sigma \in \mathbf{T}$ for all $t\in \mathbf{T}$, so we can define the character $\chi^\sigma$ by sending $t\in \mathbf{T}$ to $\chi(t^\sigma)$. 
\begin{proposition}\label{prop:moduleaut}
	The action of $\mathbf{Aut}(B)$ on the torus $\mathbf{T}$ induces a map $a \colon \mathbf{Aut}(B) \to \Aut(\Phi)$ with $\ker (a) =\mathbf{T}$, where $\Aut(\Phi)$ is regarded as a constant algebraic group.
\end{proposition}
\begin{proof}
Write $\Aut_\Z(X(\mathbf{T}))$ for the automorphisms of $X(\mathbf{T})$ as a $\Z$-module.
	The action of $\mathbf{Aut}(B)$ on $X(\mathbf{T})$ induces a morphism $ \mathbf{Aut}(B) \to \Aut_\Z(X(\mathbf{T}))$ by sending each $\chi \in X(\mathbf{T})$ to $\chi^\sigma$, since conjugation by $\sigma$ is an automorphism of $\mathbf{T}$. Note that each $\sigma\in \mathbf{Aut}(B)$ stabilizes $\{\chi_\beta \mid \beta \in \Phi\}$, since non-trivial characters occurring in $B$ as a representation of $\mathbf{T}$ are permuted by the action of $\mathbf{Aut}(B)$. By using $\iota$ from the previous proposition, we get a morphism $ \mathbf{Aut}(B) \to \Aut_\Z(\Z\Phi)$ such that $\Phi$ is invariant under the image of this map. This is equivalent to a morphism $ a\colon \mathbf{Aut}(B) \to \Aut(\Phi)$.	
	Suppose now that $\sigma\in \mathbf{Aut}(B)$ is in the kernel of $a$. Then $\chi_\alpha^\sigma = \chi_\alpha$ for all $\alpha\in \Phi$. Since $V_\alpha$ is $1$-dimensional, this means that $e_\alpha^\sigma = \lambda_\alpha e_\alpha$ and $f_\alpha^\sigma = \mu_\alpha f_\alpha$ for certain $\mu_\alpha,\lambda_\alpha \in k^\times$, for all $\alpha \in \Phi^+$. Because $\sigma\in \mathbf{Aut}(B)$, this means $1_\alpha^\sigma = \lambda_\alpha\mu_\alpha 1_\alpha$. But $1_\alpha$ is an idempotent in $B$, so this is only possible if $\mu_\alpha = \lambda_\alpha^{-1}$. By \cref{cor:torusdim}, this implies $\sigma\in \mathbf{T}$. Hence the kernel of $a$ is equal to $\mathbf{T}$.
\end{proof}
Since the automorphism group of $\Phi$ is relatively small, this gives us enough information to determine the entire automorphism group of $B$.
\begin{lemma}\label{lem:section}
	Let $\rho \in \Aut(\Phi)$ be an automorphism of the root system $\Phi$. Then $\rho$ induces an automorphism of $M(3^n\colon W)$ by sending each axis $(\varepsilon\alpha , \sigma_{\alpha})$ to the axis $(\varepsilon\rho(\alpha) , \sigma_{\rho(\alpha)})$. 
\end{lemma}
\begin{proof}
	This follows immediately from the fact that any automorphism of $\Phi$ induces an automorphism of the Fischer space of $3^n\colon W$. Indeed, we have  \[o((\varepsilon\alpha , \sigma_{\alpha})(\varepsilon\beta , \sigma_{\beta})) = o((\varepsilon\rho(\alpha) , \sigma_{\rho(\alpha)})(\varepsilon\rho(\beta) , \sigma_{\rho(\beta)})),\] since this is $2$ when $\alpha^\vee(\beta) = \rho(\alpha)^\vee(\rho(\beta))= 0$ and $3$ when $\alpha^\vee(\beta) = \rho(\alpha)^\vee(\rho(\beta)) = \pm 1$. 
\end{proof}

\begin{theorem}\label{thm:autgroup3nW}
	The automorphism group of $B$ is isomorphic to the group $\mathbf{T} \rtimes \Aut(\Phi)$, where the action of $\Aut(\Phi)$ on $X(\mathbf{T})$ is the geometric action on the root lattice.
\end{theorem}
\begin{proof}
	We will prove the morphism $a$ from \cref{prop:moduleaut} is an epimorphism with a section. From \cref{prop:modelB}, we know that $B\cong M(3^n\colon W)$. By \cref{lem:section}, there is a map $s\colon \Aut(\Phi) \to \mathbf{Aut}(M(3^n\colon W))\cong \mathbf{Aut}(B)$. One can verify easily that $s$ is a section for $a$, since $V_{\pm \alpha}^{s(\rho)} = V_{\pm \rho(\alpha)}$. Hence $\mathbf{Aut}(B)$ is a split extension of $\mathbf{T}$ and $\Aut(\Phi)$.
\end{proof}

\bibliographystyle{plain}
\bibliography{sources.bib}	
\end{document}